\documentclass[11pt]{amsart}

\usepackage{amssymb}
\usepackage{amsmath}
\usepackage{amsthm}

\newtheorem{thm}{Theorem}

\newtheorem{cor}{Corollary}

\begin{document}
\title{Congruences for Generalized Frobenius Partitions with Nonzero Row Difference}
\author{Kelsey Scott}
\maketitle
\section{Introduction}
We wish to extend Andrews' \cite{AndrewsGFP} general principle to include arrays with nonzero row difference and establish some congruences for these arrays. If $f_A(z,q)=f_A(z)=\sum P_A(m,n)z^{m}q^n$
denotes the generating function for the number of ordinary partitions of $n$ into $m$ parts subject to the set of restrictions $A$, then the coefficient of $z^{\alpha}$ in
\begin{equation*}
f_A(zq)f_B(z^{-1})=\sum P_A(m_1,n_1)z^{m_1}q^{m_1+n_1}\sum P_B(m_2,n_2)z^{-m_2}q^{n_2}
\end{equation*}
is the generating function
\begin{equation*}
\Phi_{A,B,\alpha}(q)=\sum_{n=0}^{\infty}\phi_{A,B,\alpha}(n)q^n
\end{equation*}
where $\phi_{A,B,\alpha}(n)$ is the number of arrays of weight $n=m_1+n_1+n_2$ wherein $(a_1,a_2,...,a_{m_1})$ is a partition of $n_1$ into $m_1$ parts arranged in nonincreasing order and subject to the set of restrictions $A$, and $(b_1,b_2,...,b_{m_2})$ is a partition of $n_2$ into $m_2$ parts arranged in nonincreasing order and subject to the set of restrictions $B$. We call $\alpha=m_1-m_2$ the \emph{row difference}. 

\section{Generating Functions}
It follows directly from the extension of Andrews' principle that the coefficient of $z^{\alpha}$ in 
\begin{equation*}
\prod_{\lambda_i=0}^{\infty}(1+zq^{({\lambda_i}+1)}+...+z^kq^{k({\lambda_i}+1)})(1+z^{-1}q^{\lambda_i}+...+z^{-k}q^{k{\lambda_i}})
\end{equation*}
is 
\begin{equation*}
\Phi_{k,\alpha}(q)=\sum_{n=0}^{\infty}\phi_{k,\alpha}(n)q^n
\end{equation*}
where $\phi_{k,\alpha}(n)$ is the number of arrays of weight $n$ and row difference $\alpha$ such that each nonnegative integer is repeated at most $k$ times in each row. 
Furthermore, the coefficient of $z^{\alpha}$ in
\begin{equation*}
\prod_{i=0}^{\infty}(1+zq^{\lambda_i+1})^k(1+z^{-1}q^{\lambda_i})^k
\end{equation*}
is 
\begin{equation*}
C\Phi_{k,a}(q)=\sum_{n=0}^{\infty}c\phi_{k,a}(n)q^n
\end{equation*}
where $c\phi_{k,\alpha}(n)$ is the number of arrays of weight $n$ and row difference $\alpha$ such that each nonnegative integer is taken from one of $k$ copies of the integers. We can index which copy of the integers an entry is taken from by thinking of this as a coloring and denoting the color by a subscript from $\{1,2,..,k\}$.

Next, we establish the formulas for the single variable generating functions with fixed row difference.

\begin{thm}\label{singlevar_phi}  
For any positive integer $k$ and any integer $\alpha$, 
\begin{align*} 
\Phi_{k,\alpha}(q)&=\frac{1}{(q;q)^k_\infty}\sum_{m_1,m_2,...,m_{k-1}=-\infty}^{\infty}(-1)^{\alpha}\zeta^{m_1(1-k)+m_2(2-k)+\cdots+m_{k-1}(-1)+k\alpha}\\ \nonumber
& \times q^{{m_1}^2+{m_2}^2+\cdots+{m^2_{k-1}}-\alpha(m_1+m_2+\cdots+m_{k-1})+\frac{\alpha^2+\alpha+\sum_{1\leq i<j\leq k-1}m_im_j}{2}}.
\end{align*}
where $\zeta=e^{\frac{2i\pi}{k+1}}$.
\end{thm}

\begin{proof}
We have that $\Phi_{k,\alpha}(q)$ is the coefficient of $z^\alpha$ in 
\begin{equation} \nonumber 
\varepsilon(z,q)=\prod_{\lambda_i=0}^{\infty}(1+zq^{\lambda_i+1}+\cdots+z^kq^{k(\lambda_i+1)})(1+z^{-1}q^{\lambda_i}+\cdots+z^{-k}q^{k\lambda_i}).
\end{equation}

From Andrews \cite{AndrewsGFP},
\begin{equation} \nonumber
\varepsilon(z, q) =\frac{1}{(q;q)^k_\infty}\prod_{j=1}^{k}\sum_{m_j=-\infty}^{\infty}(-1)^{m_j}q^{\binom{m_j+1}{2}}z^{m_j}\zeta^{jm_j}.
\end{equation}
where $\zeta=e^{\frac{2\pi{i}}{k+1}}$. 
Then, to get the coefficient of $z^\alpha$, we set $m_1+m_2+\cdots+m_k=\alpha$, and solving for $m_k$ we obtain the desired formula.
\end{proof}

\begin{thm} \label{singlevar_cphik}
For any positive integer $k$ and any integer $\alpha$, 
\begin{equation*} 
C\Phi_{k,\alpha}(q)=\frac{\sum_{m_1,m_2,...,m_{k-1}=-\infty}^{\infty}q^{Q(m_1,m_2,...,m_{k-1})}}{(q;q)^k_{\infty}}
\end{equation*}
where
\begin{align*}
Q(m_1,m_2,...,m_{k-1})&={m_1}^2+{m_2}^2+\cdots+{m^2_{k-1}}\\
\nonumber &-\alpha(m_1+m_2+\cdots+m_{k-1})+\frac{\alpha^2+\alpha+\sum_{1\leq i<j\leq k-1}m_im_j}{2}.
\end{align*}
\end{thm}

\begin{proof}
We have that $C\Phi_{k,\alpha}(n)$ is the coefficient of $z^\alpha$ in
\begin{equation*}
\varphi^k(z,q)=\prod_{\lambda_i=0}^{\infty}(1+zq^{\lambda_i+1})^k(1+z^{-1}q^{\lambda_i})^k
\end{equation*}
Then, by the version of the Jacobi Triple Product in \cite{AndrewsGFP}, we have
\begin{equation*}
\varphi^k(z,q)=\frac{1}{(q;q)^k_{\infty}}\sum_{m_1,m_2,...,m_k=-\infty}^{\infty}z^{m_1+m_2+\cdots+m_k}q^{\binom{m_1+1}{2}+\binom{m_2+1}{2}+\cdots+\binom{m_k+1}{2}}.
\end{equation*}
To get the coefficient of $z^\alpha$, we set $m_1+m_2+\cdots+m_k=\alpha$. Then, 
\begin{align*}
C\Phi_{k,\alpha}(q)&=\frac{1}{(q;q)^k_{\infty}}\sum_{m_1,m_2,...,m_{k-1}=-\infty}^{\infty}q^{\binom{m_1+1}{2}+\cdots+\binom{m_{k-1}+1}{2}+\binom{a-m_{k-1}-\cdots-m_2-m_1+1}{2}} \nonumber \\
&=\frac{1}{(q;q)^k_{\infty}}\sum_{m_1,m_2,...,m_{k-1}=-\infty}^{\infty}q^{{m_1}^2+{m_2}^2+\cdots+{m^2_{k-1}}-\alpha(m_1+m_2+\cdots+m_{k-1})} \nonumber \\
&\hspace{.8in} \times q^{\frac{\alpha^2+\alpha+\sum_{1\leq i<j\leq k-1}m_im_j}{2}}. 
\end{align*}
\end{proof}

Next, we establish product formulas for $\Phi_{2,-1}(q)$ and $C\Phi_{2,-1}(q)$ using the generating functions in Theorems \ref{singlevar_phi} and \ref{singlevar_cphik}.
\begin{cor}\label{phi_2colors}
For all nonnegative integers $n$, 
\begin{equation*}
\Phi_{2,-1}(q)=\prod_{n=1}^{\infty}\frac{1}{(1-q^{2n-1})^2(1-q^{12n-8})(1-q^{12n-6})(1-q^{12n-4})(1-q^{12n})}.
\end{equation*}
\end{cor}

\begin{proof}
From Theorem \ref{singlevar_phi} we have
\begin{align*}
\Phi_{2,-1}(q)&=\frac{1}{(q;q)^2_{\infty}}\sum_{m=-\infty}^{\infty}(-1)^{-1}\zeta^{-m-2}q^{m^2+m}\\
&=\frac{-\zeta^{-2}}{(q;q)^2_{\infty}}\sum_{m=-\infty}^{\infty}\zeta^{-m}q^{m^2+m}.\\
\end{align*}
By the Jacobi Triple products, we have
\begin{align*}
\Phi_{2,-1}(q)&=-\zeta^{-2}\prod_{i=1}^{\infty}\frac{(1-q^{2i})(1+\zeta^{-1}q^{2i})(1+\zeta q^{2i-2})}{(1-q^i)^2}\\
&=-\zeta^{-2}(1+\zeta)\prod_{i=0}^{\infty}\frac{(1-q^{2i})(1+\zeta^{-1}q^{2i})(1+\zeta q^{2i})}{(1-q^i)^2}\\
&=\prod_{i=0}^{\infty}\frac{(1-q^{2i})(1-q^{2i}+q^{4i})}{(1-q^i)^2}\\
&=\Psi_2(q)
\end{align*}
where $\Psi_{2}(q)$ is defined, and the remainder of the proof is given by Drake in \cite{Drake}.
\end{proof}

\begin{cor} \label{cphi_2colors}
For all nonnegative integers $n$, 
\begin{equation*}\label{singlevar_product}
C\Phi_{2,-1}(q)=\prod_{i=1}^{\infty}\frac{(1-q^{2i})(1+q^{2i})(1+q^{2i-2})}{(1-q^i)^2}.
\end{equation*}
\end{cor}

\begin{proof}
From Theorem \ref{singlevar_cphik}, with $k=2$ and $\alpha=-1$, we have
\begin{equation*}
C\Phi_{2,-1}(q)=\frac{\sum_{m=-\infty}^{\infty}q^{m^2+m}}{\prod_{i=1}^{\infty}(1-q^i)^2}.
\end{equation*}
We apply the Jacobi Triple Product to obtain the desired formula. 
\end{proof}

\section{Congruences}
Now, we will used the generating functions established in the previous section to establish some congruences. 
\begin{thm} For every nonnegative integer $n$, 
\begin{equation*}
\phi_{2,-1}(5n+4)\equiv 0 \pmod 5.
\end{equation*}
\end{thm}

\begin{proof}
From Corollary \ref{phi_2colors},
\begin{equation*}
\Phi_{2,-1}(q)=\prod_{n=1}^{\infty}\frac{1}{(1-q^{2n-1})^2(1-q^{12n-8})(1-q^{12n-6})(1-q^{12n-4})(1-q^{12n})}.
\end{equation*}
Then,
\begin{align*}
\Phi_{2,-1}(q)&=\prod_{n=1}^{\infty}\frac{(1-q^{n})^3(1-q^{2n})(1-q^{12n-2})(1-q^{12n-10})}{(1-q^{n})^5}\\
&\equiv \prod_{n=1}^{\infty}\frac{(1-q^{n})^3(1-q^{2n})(1-q^{12n-2})(1-q^{12n-10})}{(1-q^{5n})}\pmod 5\\
&=\frac{\sum_{j=0}^{\infty}{(-1)^{j}(2j+1)q^{\binom{j+1}{2}}\prod_{n=1}^{\infty}(1-q^{2n})(1-q^{12n-2})(1-q^{12n-10})}}{\prod_{n=1}^{\infty}(1-q^{5n})}\\
&=\frac{\sum_{j=0}^{\infty}(-1)^{j}(2j+1)q^{\binom{j+1}{2}}\sum_{m=-\infty}^{\infty}q^{9m^2-3m}-q^{9m^2+9m+2}}{\prod_{n=1}^{\infty}(1-q^{5n})}\\
&=\frac{\sum_{j=0}^{\infty}(-1)^{j}(2j+1)q^{\binom{j+1}{2}}\sum_{k=0}^{\infty}a_{k}q^{k^2+k}}{\prod_{n=1}^{\infty}(1-q^{5n})}\\
\end{align*}

where $a_{k}=1$ if $k\equiv 0,2 \pmod 3$ and $a_{k}=-2$ if $k\equiv 1 \pmod 3$.

So, 
\begin{align*}
\Phi_{2,-1}(q) &\equiv \frac{\sum_{j=0}^{\infty}\sum_{k=0}^{\infty}(-1)^{j}(2j+1)a_{k}q^{\binom{j+1}{2}+k^2+k}}{\prod_{n=1}^{\infty}(1-q^{5n})} \pmod 5\\
\end{align*}

Then, we get a contribution to $q^{5n+4}$ when
\begin{equation*}
\binom{j+1}{2}+k^2+k \equiv 4 \pmod 5
\end{equation*}

or equivalently, when 
\begin{equation*}
(2j+1)^2+2(2k+1)^2 \equiv 0 \pmod 5.
\end{equation*}

So $2j+1 \equiv 2k+1 \equiv 0 \pmod 5$.
\end{proof}

\begin{thm}
For every nonnegative integer $n$, 
\begin{equation*}
c\phi_{2,-1}(5n+4)\equiv 0\pmod 5.
\end{equation*}
\end{thm}

\begin{proof}
From Corollary \ref{cphi_2colors},
\begin{equation*}
C\Phi_{2,-1}(q)=\prod_{n=1}^{\infty}\frac{(1-q^{2i})(1+q^{2i})(1+q^{2i-2})}{(1-q^i)^2}.
\end{equation*}
Then,
\begin{align*}
C\Phi_{2,-1}(q)&=\prod_{n=1}^{\infty}\frac{(1-q^i)^3(1-q^{2i})(1+q^{2i})(1+q^{2i-2})}{(1-q^i)^5}\\
&\equiv \prod_{n=1}^{\infty}\frac{(1-q^i)^3(1-q^{2i})(1+q^{2i})(1+q^{2i-2})}{(1-q^{5i})} \pmod 5\\
&=\frac{\sum_{m=0}^{\infty}(-1)^m(2m+1)q^{\binom{m+1}{2}}\sum_{j=-\infty}^{\infty}q^{j^2+j}}{\prod_{i=0}^{\infty}(1-q^{5i})}\pmod 5\\
\end{align*}

So, we get a contribution to $q^{5n+4}$ when
\begin{equation*}
\binom{m+1}{2}+j^2+j \equiv 4 \pmod 5
\end{equation*}
or equivalently, 
\begin{equation*}
(2m+1)^2+2(2j+1)^2 \equiv 0 \pmod 5
\end{equation*}

So, $(2m+1)\equiv 0 \pmod 5$, and therefore,  $c\phi_{2,-1}(5n+4) \equiv 0 \pmod 5$. 

\end{proof}

\end{document}